\documentclass[11pt]{article}

\usepackage{amsmath,amsthm,amssymb,fancyhdr,graphicx,bbm,cancel,color,mathrsfs,todonotes,hyperref,xypic, pinlabel}
\usepackage[all]{xy}

\usepackage{bm}
\newtheorem{thm}{Theorem}%[section]
\newtheorem{mainthm}{Theorem}

\newtheorem{lem}[thm]{Lemma}
\newtheorem{prop}[thm]{Proposition}

\theoremstyle{definition} %The commands below have bold title, standard text

\newtheorem{qu}[thm]{Question}

\theoremstyle{remark} %The commands below have italic title, standard text
\newenvironment{enum}
{\begin{enumerate}\vspace{-.1in}
  \setlength{\itemsep}{1pt}
  \setlength{\parskip}{0pt}
  \setlength{\parsep}{0pt}
}{\end{enumerate}}

\newcommand{\dmo}{\DeclareMathOperator}

\newcommand{\R}{\mathbb{R}}

\newcommand{\co}{\mathbb{C}}\newcommand{\Z}{\mathbb{Z}}
\newcommand{\N}{\mathbb{N}}

\newcommand{\al}{\alpha}\newcommand{\be}{\beta}\newcommand{\ga}{\gamma}\newcommand{\de}{\delta}\newcommand{\ep}{\epsilon}
\newcommand{\Om}{\Omega}\newcommand{\De}{\Delta}\newcommand{\ka}{\kappa}\newcommand{\la}{\lambda}
\newcommand{\Ga}{\Gamma}
\newcommand{\what}{\widehat}\newcommand{\wtil}{\widetilde}\newcommand{\Si}{\Sigma}\newcommand{\Lam}{\Lambda}

\newcommand{\cd}{\cdots}\newcommand{\ld}{\ldots}
\newcommand{\sbs}{\subset}\newcommand{\bbm}{\mathbbm}

\newcommand{\xra}{\xrightarrow}

\newcommand{\ra}{\rightarrow}
\newcommand{\hra}{\hookrightarrow}\newcommand{\onto}{\twoheadrightarrow}

\newcommand{\bb}[1]{\mathbb{#1}}\newcommand{\ca}[1]{\mathcal{#1}}\newcommand{\ov}[1]{\overline{#1}}\newcommand{\mf}{\mathfrak}

\newcommand{\fr}[2]{\frac{#1}{#2}}

\newcommand{\op}{\oplus}\newcommand{\til}{\tilde}
\newcommand{\ti}{\times}

\newcommand{\pair}[1]{\langle #1 \rangle}

\dmo{\sgn}{sign}
\dmo{\we}{\wedge}
\dmo{\ind}{ind}\dmo{\Ind}{Ind}
\dmo{\bop}{\bigoplus}\dmo{\pic}{Pic}
\dmo{\coker}{coker}\dmo{\vol}{Vol}\dmo{\gal}{Gal}\dmo{\perm}{Perm}
\dmo{\tor}{Tor}\dmo{\ext}{Ext}\dmo{\Ext}{Ext}
\dmo{\aut}{aut}
\dmo{\Aut}{Aut}
\dmo{\inn}{Inn}\dmo{\var}{Var}
\dmo{\dep}{depth}\newcommand{\rest}[2]{#1\bigr\vert_{#2}}
\dmo{\ad}{ad}\dmo{\curl}{curl}
\dmo{\hy}{\bb H}\dmo{\Sl}{SL}
\dmo{\SO}{SO}\dmo{\psl}{PSL}
\dmo{\isom}{Isom}\dmo{\Isom}{Isom}
\dmo{\conf}{Conf}
\dmo{\stab}{Stab}\dmo{\Jac}{Jac }
\dmo{\diam}{diam}\dmo{\fix}{Fixed}\dmo{\Fix}{Fix}
\dmo{\injR}{injRad}\dmo{\Ad}{Ad}
\dmo{\esv}{ess-vol}\dmo{\out}{Out}\dmo{\Out}{Out}
\dmo{\nil}{Nil}\dmo{\sol}{Sol}
\dmo{\Div}{div}
\dmo{\SU}{SU}
\dmo{\SP}{SP}
\dmo{\Sp}{Sp}
\dmo{\rk}{rk}
\dmo{\rank}{rank}
\dmo{\psp}{PSp}\dmo{\psu}{PSU}
\dmo{\PU}{PU}\dmo{\pgl}{PGL}
\dmo{\Mod}{Mod}\dmo{\range}{Range}
\dmo{\eu}{eu}\dmo{\mi}{mi}
\dmo{\Log}{Log}\dmo{\supp}{supp}
\dmo{\maps}{Maps}\dmo{\Gr}{Gr}
\dmo{\Pin}{Pin}
\dmo{\Spin}{Spin}\dmo{\Str}{Str}
\dmo{\Sq}{Sq}\dmo{\Symp}{Symp}
\dmo{\pd}{PD}\dmo{\PD}{PD}\dmo{\sig}{Sig}
\dmo{\Set}{Set}\dmo{\Top}{Top}
\dmo{\ev}{ev}\dmo{\St}{St}
\dmo{\Pt}{Pt}\dmo{\pt}{pt}

\dmo{\colim}{colim }\dmo{\Pl}{PL}
\dmo{\String}{String}\dmo{\smear}{smear}

\dmo{\dev}{dev}
\dmo{\met}{Met}\dmo{\contact}{Contact}
\dmo{\teich}{Teich}\dmo{\Teich}{Teich}\dmo{\qi}{QI}
\dmo{\der}{Der}
\dmo{\cl}{Cliff}\dmo{\Cl}{Cl}
\dmo{\Pf}{Pf}
\dmo{\ch}{ch}\dmo{\diag}{diag}
\dmo{\grad}{grad}\dmo{\Char}{char}
\dmo{\spec}{Spec}\dmo{\Arg}{Arg}
\dmo{\rad}{rad}\dmo{\im}{Im}
\dmo{\Hom}{Hom}\dmo{\End}{End}
\dmo{\tr}{tr}\dmo{\id}{Id}
\dmo{\gl}{GL}
\dmo{\sym}{Sym}\dmo{\Sym}{Sym}
\dmo{\com}{Comm}
\dmo{\Lk}{Lk}
\dmo{\CAT}{CAT}
\dmo{\Rep}{Rep}
\dmo{\Conf}{Conf}
\dmo{\PConf}{PConf}
\dmo{\Push}{Push}
\dmo{\Cont}{Cont}
\dmo{\sm}{\setminus}
\dmo{\vn}{\varnothing}
\dmo{\disk}{\mathbb D}
\dmo{\Trd}{Trd}\dmo{\Mat}{Mat}
\dmo{\Riem}{Riem}
\dmo{\Diffn}{\Diff_0}\dmo{\diff}{diff}
\dmo{\Diff}{Diff}\dmo{\homeo}{Homeo}
\dmo{\Homeo}{Homeo}\dmo{\Fr}{Fr}
\dmo{\rot}{rot}\dmo{\Emb}{Emb}
\dmo{\Ham}{Ham}\dmo{\Met}{Met}
\dmo{\Ein}{Ein}\dmo{\CP}{\co P}
\dmo{\Per}{Per}\dmo{\Ric}{Ric}
\newcommand{\C}{\mathbb C}\dmo{\Nrd}{Nrd}
\dmo{\Comp}{Comp}\dmo{\PSC}{PSC}
\dmo{\Cent}{Cent}\dmo{\Orb}{Orb}
\dmo{\aind}{a-ind}\dmo{\tind}{t-ind}
\dmo{\constant}{constant}
\dmo{\Td}{Td}
\dmo{\LMod}{LMod}
\dmo{\SMod}{SMod}
\dmo{\SDiff}{SDiff}
\dmo{\Br}{Br}
\dmo{\csch}{csch}
\dmo{\triv}{triv}
\dmo{\genus}{genus}
\dmo{\Homeq}{HomEq}
\dmo{\PP}{\mathbb{P}}
\dmo{\U}{U}
\dmo{\Gal}{Gal}
\dmo{\BDiff}{\wtil{\Diff}}
\dmo{\BAut}{\wtil{\Aut}}
\dmo{\Iso}{Iso}
\dmo{\codim}{codim}
\dmo{\II}{II}
\dmo{\I}{I}
\dmo{\InjRad}{InjRad}
\dmo{\Inn}{Inn}
\dmo{\sys}{sys}
\dmo{\Comm}{Comm}
\dmo{\PO}{PO}
\dmo{\POm}{P\Om}
\dmo{\ab}{ab}
\dmo{\PSO}{PSO}

\usepackage{enumerate,pdfpages,stmaryrd} 
\usepackage[margin=0.8in]{geometry}

\usepackage{tikz}
\usepackage{dsfont}
\usetikzlibrary{matrix}

\begin{document}

\title{Symmetries of exotic negatively curved manifolds}

\author{Mauricio Bustamante and Bena Tshishiku}
%\address{Institut f\"ur Mathematik, Universit\"at Augsburg} \email{bustamante.math@gmail.com}

%\author{Bena Tshishiku}
%\address{Department of Mathematics, Harvard University} \email{tshishikub@gmail.com}

%\subjclass[2000]{Primary 54C40, 14E20; Secondary 46E25, 20C20}%    General info

\date{}

%\keywords{Aspherical manifolds, negative curvature, exotic spheres, Nielsen realization, group actions}

\maketitle

%\tableofcontents

\begin{abstract}
Let $N$ be a smooth manifold that is homeomorphic but not diffeomorphic to a closed hyperbolic manifold $M$. In this paper, we study the extent to which $N$ admits as much symmetry as $M$. Our main results are examples of $N$ that exhibit two extremes of behavior. On the one hand, we find $N$ with maximal symmetry, i.e.\ $\Isom(M)$ acts on $N$ by isometries with respect to some negatively curved metric on $N$. For these examples, $\Isom(M)$ can be made arbitrarily large. On the other hand, we find $N$ with little symmetry, i.e.\ no subgroup of $\Isom(M)$ of ``small" index acts by diffeomorphisms of $N$. The construction of these examples incorporates a variety of techniques including smoothing theory and the Belolipetsky--Lubotzky method for constructing hyperbolic manifolds with a prescribed isometry group. \end{abstract}

\section{Introduction}

Throughout this paper, $M=\hy^n/\pi$ denotes a closed hyperbolic manifold with fundamental group $\pi$, and $N$ denotes an \emph{exotic smooth structure} (on $M$), i.e.\ a smooth manifold that is homeomorphic but not diffeomorphic to $M$. Define the \emph{symmetry constant} of $N$ as the supremum 
\[s(N)=\sup_\rho \fr{|\Isom(N,\rho)|}{|\Isom(M)|},\]
over all Riemannian metrics $\rho$ on $N$. In this paper we study the possible values of this invariant. There is an ``easy" bound
\begin{equation}\label{eqn:easy-bound}\fr{1}{|\Isom(M)|}\le s(N)\le 1\end{equation}
that follows from Mostow rigidity and a theorem of Borel (explained below). Our main results follow: 

\begin{mainthm}[maximal symmetry constant]\label{thm:symmetric}
Fix $n$ such that the group $\Theta_n$ of exotic spheres is nontrivial. For every $d>0$, there exists a closed hyperbolic manifold $M^n$ and an exotic smooth structure $N$ such that $|\Isom(M)|\ge d$ and $s(N)=1$. 
\end{mainthm}

\begin{mainthm}[arbitrarily small symmetry constant]\label{thm:asymmetric} 
Fix $n$ such that $\Theta_{n-1}\neq0$. For every $d>1$, there exists a closed hyperbolic manifold $M^n$ and an exotic smooth structure $N$ such that $s(N)\le\fr{1}{d}$. 
\end{mainthm}

The hypothesis $\Theta_n\neq0$ is frequently true, e.g.\ $\Theta_{4k+3}\neq0$ for every $k\ge1$ and $\Theta_{4k+1}$ is nontrivial for any positive $k\notin\{1,3,7,15,31\}$. See \cite[\S7]{kervaire-milnor}, \cite[Appx.\ B]{milnor-stasheff}, and \cite[Thm.\ 1.3]{hill-hopkins-ravenel}. 

The problem of computing $s(N)$ is related to two different problems in the study of transformation groups: 
\begin{itemize} 
\item {\it Degree of symmetry.} The degree of symmetry $\de(W)$ of a manifold $W$ is defined as the largest dimension of a compact Lie group with a smooth, effective action on $W$ \cite{hsiang-hsiang-degree-symmetry}. 

When $W=\Si$ is an exotic sphere, computing $\de(\Si)$ is equivalent to computing the supremum \[s(\Si):=\sup_\rho \fr{\dim\Isom(\Si,\rho)}{\dim\Isom(S^n)},\]  over all Riemannian metrics $\rho$. Again there is a bound $\fr{1}{\dim\SO(n+1)}\le s(\Si)\le 1$, but the upper bound is not optimal. For example, Hsiang--Hsiang \cite{hsiang,hsiang-hsiang} prove that if $\Si\neq S^n$ has dimension $n\ge40$, then $s(\Si)<\fr{n^2+8}{4(n^2+n)}<1/4$.

When $W$ is an aspherical manifold and $\pi_1(W)$ is centerless, then $\de(W)=0$, i.e.\ $W$ does not admit a nontrivial action of a connected Lie group \cite{borel-isometry-aspherical}. In this case it's fitting to define $\de(W)$ as the largest order of a finite group that acts effectively on $W$. With this definition, for $W=N$ an exotic smooth structure on a hyperbolic manifold, $\de(N)$ is closely related to $s(N)$; see equation (\ref{eqn:nielsen-constant}) below.

\item {\it Propagating group actions} \cite{adem-davis}. 
One says that an $F$-action on $Y$ \emph{propagates across} a map $f:X\ra Y$ if there is an $F$-action on $X$ and an equivariant map $X\ra Y$ that is homotopic to $f$. In particular, for an exotic smooth structure $N$ on a hyperbolic manifold $M$, and for a subgroup $F<\Isom(M)$, one can ask whether or not the action of $F$ propagates across some homeomorphism $N\ra M$. This problem, and its relation to harmonic maps, is discussed in Farrell--Jones \cite{FJ-nielsen}. Theorems \ref{thm:symmetric} and \ref{thm:asymmetric} can be viewed as positive and negative results about propagating group actions, and give partial answers the question of \cite[pg.\ 487]{FJ-nielsen}.  
\end{itemize} 

%{\it Remark.} Talk about the problem of propagating group actions (in the study of transformation groups \cite{adem-davis}) and how these results are related to that. 

{\it Remark.} One could consider refinements of the symmetry constant such as $s_{<0}(N)=\sup_\rho\fr{|\Isom(N,\rho)|}{|\Isom(M)|}$, where the supremum is over all metrics with sectional curvature $K<0$. In general, $s_{<0}(N)\le s(N)$, but computing $s_{<0}(N)$ is more difficult (e.g.\ it does not reduce to a Nielsen realization problem; see below). We improve upon Theorem \ref{thm:symmetric} by giving examples for which $s_{<0}(N)=s(N)=1$. %In fact, the hardest part of Theorem \ref{thm:symmetric} is producing \emph{negatively curved} $(N,\rho)$ with isometric action of $\Isom(M)$. 

\begin{mainthm}[maximal symmetry, achieved by negatively-curved metric]\label{thm:isom-split}
Fix $n$, and assume that either $n$ is even or $|\Theta_n|$ is not a power of $2$. Given $d>0$, there exists a closed hyperbolic manifold $M^n$ and an exotic smooth structure $N$ such that $|\Isom(M)|\ge d$ and $N$ admits a Riemannian metric $\rho$ with negative sectional curvature so that $\Isom(N,\rho)\simeq\Isom(M)$. 
\end{mainthm}

If $n=4k+3$, then $|\Theta_n|$ is divisible by $2^{2k+1}-1$; see 
 \cite[Appx.\ B]{milnor-stasheff}. 

\subsection{Techniques} 
The problem of determining $s(N)$ is related to a \emph{Nielsen realization problem}, which will be our main point of view. By Borel \cite{borel-isometry-aspherical} any compact Lie group that acts effectively on $N$ is finite; furthermore, any finite subgroup of $\Diff(N)$ acts faithfully on $\pi=\pi_1(N)$. Consequently, for every $\rho$, the isometry group $\Isom(N,\rho)$ is a subgroup of $\Out(\pi)=\Aut(\pi)/\pi$. Furthermore, if $\dim M\ge3$, then $\Out(\pi)\simeq\Isom(M)$ by Mostow rigidity. This explains the upper bound in (\ref{eqn:easy-bound}). A subgroup $F< \Out(\pi)$ is said to be \emph{realized by diffeomorphisms} when can we solve the lifting problem (commonly called the Nielsen realization problem --- see e.g.\  \cite{block-weinberger} and \cite{mann-tshishiku}): 
\[\begin{xy}
(0,0)*+{\Diff(N)}="A";
(0,-15)*+{\Out(\pi)}="B";
(-25,-15)*+{F}="C";
{\ar"A";"B"}?*!/_4mm/{\Psi_N};
{\ar@{-->} "C";"A"}?*!/_3mm/{};
{\ar@{^{(}->} "C";"B"}?*!/_2mm/{};
\end{xy}\]

%The homomorphism $\Psi_N$ sends a diffeomorphism its induced automorphism of $\pi=\pi_1(N)$, which is well-defined up to inner automorphisms. 
If $F<\Out(\pi)$ and $F\simeq\Isom(N,\rho)$ for some $\rho$, then group $F$ is a fortiori realized by diffeomorphisms. Conversely, if $F<\Out(\pi)$ is realized by diffeomorphisms, then by averaging a metric, we find $\rho$ with $F<\Isom(N,\rho)$. Therefore, 
\begin{equation}\label{eqn:nielsen-constant}s(N)=\max_F\fr{|F|}{|\Out(\pi)|},\end{equation} where the maximum is over the subgroups $F<\Out(\pi)$ that are realized by diffeomorphisms. Note that $s(N)\le\fr{|\im\Psi_N|}{|\Out(\pi)|}$. 

Farrell--Jones \cite{FJ-nielsen} studied the Nielsen realization problem for  $N=M\#\Si$, where $M^n$ is a closed, oriented hyperbolic manifold and $\Si\in\Theta_n$ is a nontrivial exotic sphere. The main result of \cite{FJ-nielsen} states that if $M$ is stably parallelizable, $2\Si\neq0$ in $\Theta_n$, and $M$ admits an orientation-reversing isometry, then $\im\Psi_N<\Out(\pi)$ has index at least 2. In particular, $s(N)\le 1/2$ for these examples.  

\vspace{.1in} 
{\bf Symmetric exotic smooth structures.} Here we discuss the main components in the proof of Theorems \ref{thm:symmetric} and \ref{thm:isom-split}. We find our examples with $s(N)=1$ among the manifolds $N=M\#\Si$ studied by Farrell--Jones. Using (\ref{eqn:nielsen-constant}), observe that $s(N)=1$ if and only if $\Out(\pi)$ is realized by diffeomorphisms of $N$. In particular, we must find examples where $\Psi_N$ is surjective. The following results refine \cite[Thm.\ 1]{FJ-nielsen}. 

\begin{thm}\label{thm:image}
Let $M^n$ be a closed, oriented hyperbolic manifold, let $\Si\in\Theta_n$ be a nontrivial exotic sphere, and let $N=M\#\Si$. Denote by $\Out^+(\pi)<\Out(\pi)$ the subgroup that acts trivially on $H_n(N)\simeq\Z$. 
\begin{enum}
\item[(a)] The image $\im\Psi_N$ contains $\Out^+(\pi)$. 
\item[(b)] Fix $\alpha\in\Out(\pi)\setminus\Out^+(\pi)$. If $2\Si=0$ in $\Theta_n$, then $\al\in\im\Psi_N$. The converse is true if $M$ is stably parallelizable. 
\end{enum} 
\end{thm}

Every closed hyperbolic manifold has a finite cover that is stably parallelizable \cite[pg.\ 553]{sullivan-stably-parallelizable}. As a consequence of Theorem \ref{thm:image}, if $2\Si=0$, then $\Psi_N$ is surjective, and if $2\Si\neq0$, then $\im\Psi_N=\Out^+(\pi)$. In any case, if $M$ does not admit an orientation-reversing isometry, then $\Psi_N$ is surjective. Farrell--Jones \cite{FJ-exotic-hyperbolic} show (implicitly) that reversing orientation is an obstruction to belonging to $\im\Psi_N$ when $2\Si\neq0$. According to Theorem \ref{thm:image}, this is the only obstruction. 

Having identified $\im\Psi_N<\Out(\pi)$, we would like to know if this subgroup is realized by diffeomorphisms. 

\begin{thm}\label{thm:diff-split}
Fix $N=M\#\Si$ as in Theorem \ref{thm:image}. Set $d=|\Isom^+(M)|$ and let $m\in\N$ be the size of the largest cyclic subgroup of $\Theta_n$ that contains $\Si$. Assume that $\gcd(d,m)$ divides $\fr{m}{|\Si|}$. Then $\Out^+(\pi)$ is realized by diffeomorphisms. 
\end{thm}

The assumption $\gcd(d,m)\mid \fr{m}{|\Si|}$ guarantees that $\Si\in\Theta_n$ has a $d$-th root. This condition is satisfied, for example, whenever $|\Isom^+(M)|$ and $|\Si|$ are relatively prime. 

If $\Out^+(\pi)$ is realized by diffeomorphisms of $N$, then $s(N)\ge 1/2$. By Theorems \ref{thm:image} and \ref{thm:diff-split}, if $M$ is stably parallelizable and $2\Si\neq0$, then $s(M\#\Si)$ is equal to $1/2$ or $1$, according to whether or not $M$ admits an orientation-reversing isometry. This completely solves the Nielsen realization problem in these cases.

Theorem \ref{thm:symmetric} reduces to Theorem \ref{thm:diff-split}. Fixing $\Si\neq S^n$, it's possible to find $M$ so that $|\Isom^+(M)|$ and $|\Si|$ are relatively prime, and $|\Isom^+(M)|$ can be made arbitrarily large. This is a consequence of a result of Belolipetsky--Lubotzky \cite{belolipetsky-lubotzky}: for any finite group $F$, there exists a closed hyperbolic $M^n$ with $\Isom(M)=F$. For their examples $\Isom(M)=\Isom^+(M)$. In particular, one can find examples where $\Psi_N:\Diff(N)\ra\Out(\pi)$ is a split surjection with $|\Out(\pi)|$ arbitrarily large.

%\begin{cor}Fix $N=M\#\Si$ as in Theorem \ref{thm:diff-split}. Assume $M$ is stably parallelizable. If $2\Si\neq0$, then $s(N)$ is equal to $1/2$ or $1$, according to whether or not $M$ admits an orientation-reversing isometry. 
%\end{cor}

To prove Theorem \ref{thm:isom-split}, one would like to promote the action of $\Out^+(\pi)$ on $N=M\#\Si$ produced in Theorem \ref{thm:diff-split} to an action by isometries with respect to some negatively curved metric on $N$. Using a warped-metric construction of Farrell--Jones \cite{FJ-exotic-hyperbolic}, it suffices to find an $M$ that is stably parallelizable, has large injectivity radius, and such that $\Isom^+(M)$ acts freely on $M$. Arranging all of these conditions simultaneously becomes delicate, especially arranging that $M$ is stably parallelizable (which is desired because it guarantees that $M\#\Si$ is not diffeomorphic to $M$). Because of this difficulty we take a less direct approach when $\dim M$ is odd --- see Theorem \ref{thm:BL+}.

%Farrell--Jones \cite{FJ-exotic-hyperbolic} proved that $N$ admits a Riemannian metric of negative curvature, at least after replacing $M$ with a finite cover with large injectivity radius. The following strengthening of Theorem \ref{thm:diff-split} is an elaboration of Theorem \ref{thm:isom-split}. 

%
%\begin{thm}\label{thm:isom-split}
%Fix $\Si\neq 0\in\Theta_n$, and let $G$ be a finite group so that $\gcd(|G|,m)$ divides $\fr{m}{|\Si|}$, where $m$ is defined as in Theorem \ref{thm:diff-split}. Among the hyperbolic $M^n$ with $\Out(\pi)=\Out^+(\pi)=G$, there exists $M$ that is stably parallelizable and so that $N=M\#\Si$ admits a  Riemannian metric $\rho$ with negative sectional curvature and $\Isom(N,\rho)\simeq G$. 
%\end{thm}

%In order to put a negatively curved metric on $N$, we use the Farrell--Jones warped-product metrics \cite{FJ-exotic-hyperbolic}. For this one requires the injectivity radius of $M$ to be large. 
%One interesting by-product of the proof of Theorem \ref{thm:isom-split} is that we show that for any finite group $G$, one can arrange simultaneously that $\Isom(M)\simeq G$ and that $M$ has arbitrarily large injectivity radius. See Theorem \ref{thm:BL+}. 

{\bf Asymmetric exotic smooth structures.} We explain the main ideas for proving Theorem \ref{thm:asymmetric}. For this, we consider exotic smooth structures $N=M_{c,\phi}$ obtained by removing a tubular neighborhood $S^1\times D^{n-1}\hra M$ of a geodesic $c\sbs M$ and gluing in $S^1\ti D^{n-1}$ by a diffeomorphism $\bbm1\ti\phi$ of $S^1\times S^{n-2}$, where $\phi\in\Diff(S^{n-2})$ is not isotopic to the identity. Farrell--Jones \cite{FJ-nonuniform} prove that $M_{c,\phi}$ is often an exotic smooth structure on $M$. 

The strategy for proving Theorem \ref{thm:asymmetric} is to find $N=M_{c,\phi}$ and $F\simeq\Z/d\Z$ in $\Out(\pi)$ so that $\im\Psi_N\cap F=1$. 
This condition implies that the index of $\im\Psi_N<\Out(\pi)$ is at least $|F|$, so $s(N)\le \fr{1}{|F|}$. To show $F\cap\im\Psi_N=1$, we study how the smooth structure on $M_{c,\phi}$ changes if we choose a different geodesic $c$. This is complementary to \cite[Thm.\ 1.1]{FJ-nonuniform}, which studies how the smooth structure changes when the geodesic is fixed and the isotopy class $[\phi]\in\pi_0\Diff(S^{n-2})\simeq\Theta_{n-1}$ is changed. In Theorem \ref{thm:non-concordant} we give a criterion to guarantees that $M_{c_1,\phi}$ and $M_{c_2,\phi}$ are not \emph{concordant}, i.e.\ there is no smooth structure on $M\times[0,1]$ that restricts to $M_{c_1,\phi}\sqcup M_{c_2,\phi}$ on the boundary. This is one of the main technical ingredients in the proof of Theorem \ref{thm:asymmetric}.

%\begin{mainthm}\label{thm:large-index}
%For every $d>0$ there exists a hyperbolic manifold $M$ and an exotic smooth structure $N$ so that the index of $\im\Psi_N$ in $\Out(\pi)$ is at least $d$. 
%\end{mainthm}

The proof of Theorem \ref{thm:asymmetric} works equally well when $M$ is nonuniform, but we won't discuss this further. 

Theorem \ref{thm:asymmetric} proves that $s(N)$ may be arbitrarily close to 0, as $N$ varies over exotic smooth structures on all hyperbolic $n$-manifolds (when $\Theta_{n-1}\neq0$), but if we fix the homeomorphism type, we know that $s(N)\ge\fr{1}{|\Isom(M)|}$. It would be interesting to know if there are examples where this lower bound is achieved. Of course if $\Isom(M)=1$, then $s(N)=1=\fr{1}{|\Isom(M)|}$, so to make this interesting one should ask for examples such that $\Isom(M)$ is large.  

\begin{qu}
Does there exist $n$ so that for every $d>0$, there exists a hyperbolic manifold $M^n$ and an exotic smooth structure $N$ such that $|\Isom(M)|\ge d$ and $s(N)=\fr{1}{|\Isom(M)|}$? 
\end{qu}

Note that $s(N)=\fr{1}{|\Isom(M)|}$ if and only if $\Psi_N:\Diff(N)\ra\Out(\pi)$ is trivial. Equivalently, $\Isom(N,\rho)=1$ for every Riemannian metric $\rho$. 

{\it Section outline.} In \S\ref{sec:proofs1} we prove Theorems \ref{thm:image} and \ref{thm:diff-split} and discuss some related questions of interest. In \S\ref{sec:isom-split} we discuss the work of Belolipetsky--Lubotzky and use it to prove Theorem \ref{thm:isom-split}. Finally, in \S\ref{sec:asymmetric} we prove Theorem \ref{thm:asymmetric}; specifically, we study when two smooth structures $M_{c_1,\phi}$ and $M_{c_2,\phi}$ are concordant, which we use as an obstruction to Nielsen realization.

{\it Acknowledgements.} The authors would like to thank I.\ Belegradek and S.\ Cappell for helpful and interesting conversations. M.B. has 
been supported by the Special Priority Program SPP 2026 ``Geometry at Infinity” funded by the Deutsche Forschungsgemeinschaft (DFG).

\section{Symmetry constant for $N=M\#\Si$}\label{sec:proofs1}

In this section we prove Theorems \ref{thm:image} and \ref{thm:diff-split}. 

\subsection{The image of $\Psi_N:\Diff(N)\ra\Out(\pi)$}\label{sec:image}

\begin{proof}[Proof of Theorem \ref{thm:image}] 

Let $N=M\#\Si$ as in the theorem. It will be convenient to fix $p\in M$ and a small metric ball $B=B_r(p)$ where the connected sum is performed. 

First we prove (a). For this we fix $\al\in\Out^+(\pi)\simeq\Isom^+(M)$ and define $f\in\Diff(N)$ so that $\Psi_N(f)=\al$. View $\al$ as an isometry of $M$, and choose an isotopy $\al_t\in\Diff(M)$ so that $\al_0=\al$ and $\al_1(B)=B$ and $\al_1\rest{}{B}\in O(n)$ is an isometry of the ball; for example, if the radius $r$ is sufficiently small, then we can isotope $\al(B)$ to $B$ in $M$ through isometric embeddings, and then extend the isotopy of $B$ to an ambient isotopy. Since $\al$ is orientation-preserving, $\al_1\rest{}{B}$ belongs to the identity component $SO(n)\sbs O(n)$, and it is easy to see then that $\al_1$ induces a diffeomorphism $f:N\ra N$; for example, isotope $\al_1\rest{}{B}$ further so that $\al_1\rest{}{B_{r/2}(p)}$ is the identity and perform the connected sum along $B_{r/2}(p)$ instead of $B_r(p)$. This proves part (1). 

To prove (b), assume that $\alpha\in\Out(\pi)\setminus\Out^+(\pi)$. Viewing $\alpha$ as an orientation-reversing isometry of $M$, the argument above defines an orientation-reversing diffeomorphism $h:M\#\Si\ra M\#\ov \Si$ that induces $\alpha$ (recall that for $A\#B$, if the identification of the attaching disk is changed by an orientation-reversing involution, then the result is $A\#\ov B$, where $\ov B$ is $B$ with the opposite orientation). If $2\Si=0$ in $\Theta_n$, then $\Si=\ov\Si$ (because $\ov\Si=-\Si$ in $\Theta_n$), so $h\in\Diff(N)$ and $\Psi_N(h)=\alpha$. This proves the first statement of (b). The converse is already to contained in \cite[Thm.\ 1]{FJ-nielsen}. In short, if $\Psi_N(f)=\alpha$ for some $f\in\Diff(N)$, then $h\circ f$ is an orientation-preserving diffeomorphism $M\#\Si\ra M\#\ov\Si$. When $M$ is stably parallelizable, this implies that $2\Si=0$ by \cite[\S2]{FJ-exotic-hyperbolic}.
\end{proof}

\subsection{Sections of $\Psi_N:\Diff(N)\ra\im\Psi$}

\begin{proof}[Proof of Theorem \ref{thm:diff-split}]
Since $M$ is hyperbolic, $\Out(\pi)$ is realized by isometries of $M$ (by Mostow rigidity). Set $F=\Isom^+(M)$. Since $F$ is finite, there exists $p\in M$ whose stabilizer in $F$ is trivial. Choose a ball $B$ around $p$ whose $F$-translates are disjoint. By assumption, $\gcd(|F|,m)$ divides $\fr{m}{|\Si|}$, which implies that there exists $\Si'\in\Theta_n$ so that $\Si=|F|\cdot\Si'$. Then $N=M\#\Si$ is diffeomorphic to $M\#\Si'\#\cd\#\Si'$, where $\Si'$ appears $|F|$ times. If we form the connected sum along the union of balls $F.B$, then we can extend the action of $F$ on $M\setminus F.B$ to a smooth $F$-action on $N=M\#\Si'\#\cd\#\Si'$ by rigidly permuting the exotic spheres. 
\end{proof}

{\it Remark.} One might think that the above argument could be used to define an action of $\Out(\pi)$ on $N$ under a similar constraint on $|\Out(\pi)|$ and $|\Si|$. This would contradict the fact that $\Psi_N$ is frequently not surjective when $M$ admits an orientation-reversing isometry. In the argument above, when $M$ admits an orientation-reversing isometry, one obtains an action of $\Out(\pi)$ on $M\# k\Si'\# k\ov{\Si'}$, where $k=|\Out(\pi)|/2$. But $M\# k\Si'\# k\ov{\Si'}$ is diffeomorphic to $M$, not $N$.

%{\it Remark.} There is an interesting case not addressed by Theorem \ref{thm:diff-split}. Suppose $N=M\#\Si$, where $2\Si=0$ and $M$ admits an orientation reversing isometry. Then $\Psi_N:\Diff(N)\ra\Out(\pi)$ is surjective and splits over the index-2 subgroup $\Out^+(\pi)$ by Theorems \ref{thm:image} and \ref{thm:diff-split}; however, the argument of Theorem \ref{thm:diff-split} does not seem to address whether or not $\Psi_N$ splits over the full group $\Out(\pi)$. 

%\begin{prob}
%Fix $N=M\#\Si$. Assume that $M$ admits an orientation-reversing isometry $\al$ and that $2\Si=0$. Determine if the subgroup $\pair{\al}\simeq\Z/2\Z$ of $\Out(\pi)$ is realized by diffeomorphisms. In other words, among the diffeomorphisms $f:N\ra N$ such that $\Psi_N(f)=\alpha$, does there exist one with $f^2=\bbm1$? 
%\footnote{{\color{red} Does this become easier if $\al$ fixes a point? Does the exotic sphere of order 2 in $\Theta_7$ admit an order-2 orientation-reversing diffeomorphism?? } }
%\end{prob}

It would be interesting to know if $\Out^+(\pi)$ ever acts on $N=M\#\Si$ when $N$ has no ``obvious" symmetry:

\begin{qu}\label{q:no-obvious}
Is Theorem \ref{thm:diff-split} ever true without the assumption $\gcd(d,m)\mid \fr{m}{|\Si|}$? For example, fix $\al\in\Isom^+(M)$ of order $d$, and assume that $\al$ acts freely. Choose $\Si\in\Theta_n$ that does not admit a $d$-th root. Prove or disprove that the subgroup $\pair{\al}\simeq\Z/d\Z$ in $\Out^+(\pi)$ is realized by diffeomorphisms of $N=M\#\Si$. 
\end{qu}

In this direction, it would be interesting to know how the choice of $\Si$ affects the answer to Question \ref{q:no-obvious}. For instance, in the study of the symmetry constant of $\Si\in\Theta_n$, there is a marked difference between (1) the standard sphere $\Si=S^n$, (2) the nontrivial exotic spheres that bound a parallelizable manifold $\Si\in bP_{n+1}\setminus\{S^n\}$, and (3) the remaining exotic spheres $\Si\in\Theta_n\setminus bP_{n+1}$. See \cite{hsiang-hsiang-degree-symmetry}. Does this distinction play a role in Question \ref{q:no-obvious}? 

Note that the subtlety in Question \ref{q:no-obvious} disappears in the topological category: if $W$ is an aspherical manifold with $\pi_1(W)\simeq\pi$, then $\Homeo(W)\ra\Out(\pi)$ is a split surjection because $W$ and $M$ are homeomorphic by the solution of Farrell--Jones to the Borel conjecture in this case; see \cite[Cor.\ 3 in \S5]{farrell-trieste}.  %$\Diff(N)$ is replaced by $\Homeo(N)$, then $\Psi_N:\Homeo(N)\ra\Out(\pi)$ is a split surjective without the assumption on $|\Isom(M)|$ and $|\Si|$ because $M$ and $N$ are homeomorphic by the solution to the Borel conjecture in this case; see \cite[Cor.\ 3 in \S5]{farrell-trieste}. 

We mention another problem related to Question \ref{q:no-obvious}. For this, let $W^n$ be an exotic smooth structure on the torus $T^n$. There is a surjective homomorphism $\Diff^+(W)\ra\Out^+(\pi_1(W))\simeq \Sl_n(\Z)$, and whether or not this homomorphism splits is unknown. One approach to this question is to focus on maximal abelian subgroups of $\Sl_n(\Z)$ and try to use the dynamics of Anosov diffeomorphisms; see \cite[Question 1.4]{fks-anosov} and also \cite{brhw}. Alternatively, an obstruction to realizing finite subgroups $F<\Sl_n(\Z)$ as in Question \ref{q:no-obvious} could provide an approach to the splitting problem for certain $W=T^n\#\Si$.

\section{Realization by isometries}\label{sec:isom-split}

In this section, we prove Theorem \ref{thm:isom-split}. The starting point of our argument is the following result from \cite[Thm.\ 1.1 and \S6.3]{belolipetsky-lubotzky}. 

\begin{thm}[Belolipetsky--Lubotzky]\label{thm:BL}
For every $n\ge2$ and every finite group $F$, there exists infinitely many compact $n$-dimensional hyperbolic manifolds $M$ with $\Isom(M)=\Isom^+(M)\simeq F$. 
\end{thm}

The main result we prove here is as follows. 

\begin{thm}\label{thm:BL+}
Fix a finite group $F$ and fix $R>0$. Among the hyperbolic manifolds $M^n$ with $\Isom(M)=\Isom^+(M)\simeq F$, there exists $M$ such that
\begin{enum}
\item[(a)] the group $F$ acts freely on $M$, 
\item[(b)] there is a cover $\what M\ra M$ of degree $\ell\in\{1,2,4\}$ so that $\what M$ is stably parallelizable, and
\item[(c)] $\InjRad(M)>R$. 
\end{enum} 
Furthermore, for (b), if $n$ is even, then we can take $\ell=1$.
\end{thm}

Next we deduce Theorem \ref{thm:isom-split} from Theorem \ref{thm:BL+}. 

\begin{proof}[Proof of Theorem \ref{thm:isom-split}]

Fix $d>0$. If $n$ is even, take any nontrivial $\Si\in\Theta_n$ and let $F$ be a group with $|F|\ge d$ and $\gcd(|F|,|\Si|)=1$. If $|\Theta_n|\neq 2^i$, take $\Si\in\Theta_n$ nontrivial of odd order and let $F$ be a $2$-primary group with $|F|\ge d$. In either case, there exists $\Si'\in\Theta_n$ with $\Si=|F|\cdot\Si'$. 
By Belolipetsky--Lubotzky and Theorem \ref{thm:diff-split}, for every $M$ with $\Isom(M)\simeq\Isom^+(M)\simeq F$, the group $F$ acts by diffeomorphisms of $N=M\#\Si\simeq M\#\Si'\#\cd\#\Si'$. We need to show we can choose $M$ and a negatively-curved metric $\rho$ on $N$ so that $F=\Isom(N,\rho)$ in $\Diff(N)$.

According to \cite[Prop.\ 1.3]{FJ-exotic-hyperbolic}, there is a constant $\tau_n>0$ so that if $M^n$ has injectivity radius $\InjRad(M)>\tau_n$, then $N=M\#\Si$ admits a negatively curved metric. This metric agrees with the hyperbolic metric on $M$ away from the disk where the connected sum is performed, and on that disk, the metric is radially symmetric. Choose $M$ satisfying Theorem \ref{thm:BL+} with $R=|F|\cdot\tau_n$ and such that $F$ acts freely on $M$, so the quotient $\ov{M}=M/F$ is a hyperbolic manifold. Furthermore, 
\begin{equation}\label{eqn:injrad-cover}\InjRad(\ov M)\ge \InjRad(M)/|F|>\tau_n.\end{equation} 
We prove this below. Now fix $r$ with $\tau_n<r<\InjRad(\ov M)$. From (\ref{eqn:injrad-cover}) it follows that for any ball $B=B_r(p)$ in $M$, the $F$-translates of $B$ are disjoint. Fix such a ball $B$. As in the proof of Theorem \ref{thm:diff-split}, write $\Si=|F|\cdot\Si'$ and consider $M_0=M\setminus F.B$. The manifold $N$ is obtained by gluing $\bb D^n$ to each boundary component of $M_0$ by a fixed diffeomorphism $f\in \Diff(S^{n-1})$. %(since we are performing the connected sum with $k$ copies of the same exotic sphere $\Si'$, the map $f$ is the same for each boundary component). 
Using the technique in \cite{FJ-exotic-hyperbolic}, we give $N$ a Riemannian metric $\rho$ that agrees with the hyperbolic metric on $M_0$ and is a warped-product metric on each $\bb D^n$. Since $r>\tau_n$, \cite[\S3]{FJ-exotic-hyperbolic} guarantees that the resulting metric has negative curvature. The group $F$ acts on $N$ as in Theorem \ref{thm:diff-split}, and by construction it acts by isometries for the metric $\rho$. 

Now we explain the inequality (\ref{eqn:injrad-cover}). To see the first inequality, note that $2\InjRad(M)=\sys(M)$, where $\sys(M)$ is \emph{systole}, i.e.\ the length of the shortest geodesic. Under a $d$-fold isometric cover $M\ra\ov M$, if $\ov\ga$ is a closed geodesic of $\ov M$ and $\ga\sbs M$ is a connected component of its preimage, then $\text{length}(\ga)\le d\cdot\text{length}(\ov\ga)$. It follows that $\sys(M)\le d\cdot\sys(\ov M)$. 

It remains is to show that $N$ is not diffeomorphic to $M$. When $n$ is even, then by Theorem \ref{thm:BL+} we can assume that $M$ is stably parallelizable and so $M$ is not diffeomorphic to $M\#\Si$ by Farrell--Jones \cite{FJ-exotic-hyperbolic}. In the general case, $M$ has a stably parallelizable cover of degree 2 or 4. Suppose for a contradiction that $M\#\Si$ is diffeomorphic to $M$. Lifting to the cover $\what M\ra M$, we find that $\what M\#\ell \Si$ is diffeomorphic to $\what M$. Note that $\ell\Si\neq0$ in $\Theta_n$ since $\Si$ has odd order and $\ell\in\{2,4\}$. Since $\what M$ is stably parallelizable, by \cite[Prop.\ 1.2]{FJ-exotic-hyperbolic}, we conclude that $\what M\#\ell\Si$ is not diffeomorphic to $\what M$. This is a contradiction, so $N$ is not
diffeomorphic to $M$ as desired. This completes the proof.
\end{proof}

%In \cite{FJ-exotic-hyperbolic}, Farrell--Jones show that if $M$ is hyperbolic and the connected sum $M\#\Si$ is performed at a point with injectivity radius larger than $\tau_n$, then it's possible to put a Riemannian metric on $M\#\Si$ with negative curvature. This metric agrees with the hyperbolic metric on $M$ away from the attaching disk. 

%We want to show that there is $M\in\Omega$ that is (i) stably parallelizable and (ii) has

%{\it Set things up here...}

Next we prove Theorem \ref{thm:BL+}. Fix a finite group $F$. In what follows $M=\hy^n/\pi$ will always denote one of the Belolipetsky--Lubotzky manifolds with $\Isom(M)=\Isom^+(M)\simeq F$. We have to explain why $M$ can be chosen to satisfy (a), (b), and (c). We will see that \cite[Thm.\ 2.1]{belolipetsky-lubotzky} already shows that (a) can be arranged, and that (b) can be arranged by modifying the proof of \cite[Prop.\ 2.2]{belolipetsky-lubotzky}. Part (c) requires a different, separate argument. All of these arguments involve passing to certain congruence covers, so once we explain why (a), (b), and (c) can be arranged \emph{individually}, it will be evident that they can be arranged \emph{simultaneously}. 

\vspace{.1in} 
{\bf Recollection of Belolipetsky--Lubotzky \cite{belolipetsky-lubotzky}.} 
Here we summarize the main results of \cite{belolipetsky-lubotzky}, especially the aspects needed for our proof. Let $\Ga$ be a finitely generated group. Assume that $\Delta\lhd\Ga$ is finite-index, normal, and that $\De$ surjects to a finite-rank free group: 
\[1\ra K\ra\De\ra F_r\ra 1\] for some $r\ge2$. The conjugation action of $N_\Ga(K)$ on $\De$ preserves $K$, so $N_\Ga(K)$ acts on $F_r$ by automorphisms. Let $D<N_\Ga(K)$ be the subgroup that acts on $F_r$ by inner automorphisms. With this setup, the main algebraic construction of \cite[Thm.\ 2.1]{belolipetsky-lubotzky} asserts that for any finite group $F$, there exists a finite-index subgroup $\pi<D$ with $N_\Ga(\pi)/\pi\simeq F$ (in their notation, they use $M$ instead of $K$ and $B$ instead of $\pi$). 

In the application to hyperbolic manifolds, define $\Ga$  as the commensurator $\Comm(\Lam)$ of a Gromov--Piatetski-Shapiro \cite{gps} non-arithmetic lattice $\Lam<\SO(n,1)$. By work of Mostow and Margulis, $\Comm(\Lam)$ is a maximal discrete subgroup of $\Isom(\hy^n)$, so for any $\pi<\Ga$, 
\[N_\Ga(\pi)/\pi\simeq N_{\Isom(\hy^n)}(\pi)/\pi\simeq\Isom(\hy^n/\pi).\]
Hence to find $M=\hy^n/\pi$ with $\Isom(M)\simeq F$, it suffices to find $\pi<\Ga$ with $N_\Ga(\pi)/\pi\simeq F$. 

To define $\De$, denote $G=O(n,1)$ and let $\ca O_S$ be ring of definition of $\Ga$, so $\Ga<G(\ca O_S)$. Let $\mf p\sbs\ca O_S$ be a prime ideal and denote $p\in\N$ the prime with $(p)=\mf p\cap\Z$. We only deal with prime ideals $\mf p$ where $\ca O_S/\mf p\simeq\bb F_p$. Equivalently, $p$ splits completely in $\ca O_S$; there are infinitely many such $\mf p$ by Chebotarev's theorem. Reduction mod $\mf p$ defines a map $\alpha_{\mf p}:\Ga\ra G(\ca O_S/\mf p)\simeq O_{n+1}(p)$ to an orthogonal group over $\bb F_p$. Define $\Ga(\mf p)=\ker(\ov\al_{\mf p})$, where $\ov\al_{\mf p}:\Ga\ra O_{n+1}(p)\ra\PO_{n+1}(p)$. The group $\De$ is defined as $\Lam\cap\Ga(\mf p)$. 

To ensure $\De\lhd\Ga$, we want $\Lambda\lhd\Ga$. In order to arrange this, after we've defined $\Ga$, we replace $\Lambda$ with a finite-index subgroup (still denoted $\Lambda$) so that $\Lambda\lhd\Ga$ (note that this replacement does not change $\Comm(\Lam)$). The group $\De$ surjects to a free group: By the cut-and-paste nature of the construction of \cite{gps}, $\Lam$ is either an amalgamated product or an HNN extension. For definiteness assume $\Lam=\Lam_1*_{\Lam_3}\Lam_2$. Denoting $\Om_{n+1}(p)=[O_{n+1}(p),O_{n+1}(p)]$, by strong approximation, for all but finitely many $\mf p$, the image of $\ov\alpha_{\mf p}:\Lam\ra \PO_{n+1}(p)$ contains $Q_p:=\POm_{n+1}(p)$, and the same is true for the restriction to $\Lam_1,\Lam_2$. Without loss of generality, we may assume $\im(\ov\al_{\mf p})=Q_p$ (replace $\Lambda$ by the intersection of all index-2 subgroups of $\Lambda$). Denoting $T_p=\ov\al_p(\Lam_3)$, the map $\ov\al_p$ factors through surjective maps $\Lam\xra{s} Q_p*_{T_p}Q_p\xra{t} Q_p$. Then $\De=\ker(t\circ s)$ surjects onto $\ker t$, which is a free group of rank $r\ge2$ \cite[Prop.\ 3.4]{belolipetsky-lubotzky}. 

%composition $\Lam\hra\Ga \ra Q$ is surjective and factors through surjective maps $\Lam\xra{s} R\xra{t} Q$, where either $R=Q*_TQ$ or $R=Q*_T$, and such that $\ker t\simeq F$ is a free group of rank $r\ge2$. Via this construction $\De=\ker(t\circ s)$ surjects to $\ker t=F$. 

%Since $\De<\Lambda$, there is a surjection $\De\ra F_r$ by the nature of the construction of $\Ga_0$. 
%Finally, we remark that Belolipetsky--Lubotzky show that $D:=\ker\big[N_\Ga(M)\ra\Out(F_r)\big]$ is contained in $\Ga(p):=\ker\big[\Ga\ra\PO_{n+1}(p)\big]$.

\vspace{.1in}
{\bf Proof of Theorem \ref{thm:BL+}.} Fix a finite group $F$. We use the setup of the proceeding paragraphs. In particular, $\pi<D$ will always denote a subgroup with $N_\Ga(\pi)/\pi\simeq F$, and our aim is to show that $\pi$ can be chosen in such a way that $M=\hy^n/\pi$ has properties (a), (b), and (c). 

{\bf Part (a).} By \cite[pg.\ 465]{belolipetsky-lubotzky} the group $N_\Ga(\pi)$ is contained in $D=\ker\big[N_\Ga(K)\ra\Out(F_r)\big]$, and \cite[\S5]{belolipetsky-lubotzky} shows that $D$ is contained in $\Ga(\mf p)$, which is torsion-free for $p$ large. It follows that $\Isom(M)\simeq N_\Ga(\pi)/\pi$ acts freely on $M$:  if $x\in M$ is fixed by $g\neq 1\in\Isom(M)$, then $g$ lifts to $\til g\in N_\Ga(\pi)$ that acts on $\hy^n$ with a fixed point, but this contradicts the fact that $N_\Ga(\pi)$ is torsion-free.  

%As mentioned above, 
%%The justification for part (a) is contained in \cite[\S4]{belolipetsky-lubotzky}. 
%%Furthermore, in the proof \cite[pg.\ 465]{belolipetsky-lubotzky}, t
%in this case they show that $\De$ can be chosen so that $D$ is contained in $\Ga(\mf p)$ for some prime ideal $\mf p\sbs\ca O$. 
% Consequently, without loss of generality, $N_\Ga(\pi)$ is torsion-free. Setting $M=\hy^n/\pi$, we conclude that $G=\Isom(M)\simeq N_\Ga(\pi)/\pi$ acts freely (if $x\in M$ is fixed by $g\neq 1$ in $G$, then $g$ lifts to $\til g\in N_\Ga(\pi)$ that acts on $\hy^n$ with a fixed point, but this contradicts the fact that $N_\Ga(\pi)$ is torsion-free). 

%In the application to hyperbolic manifolds, $\Ga$ is the commensurator of a Gromov--Piate

% (This is also remarked in \cite[\S2]{cappell-lubotzky-weinberger}.) 

{\bf Part (b).} As mentioned in part (a), we can arrange that $\pi<\Ga(\mf p)$. Our main task for part (b) will be to show that we can also arrange that $\pi<\Ga(\mf p)\cap\Ga(\mf q)$, where $\mf p,\mf q\sbs\ca O_S$ are prime ideas with $\ca O_S/\mf p\simeq\bb F_p$ and $\ca O_S/\mf q\simeq\bb F_q$ for distinct primes $p,q$. Before we do this, we explain why this is enough to conclude that $M=\hy^n/\pi$ has the desired stably parallelizable cover. %This argument originates in \cite[pg.\ 553]{sullivan-stably-parallelizable}, where a sketch is given. We summarize the argument to make it clear how it fits in our setting.

Suppose that $M=\hy^n/\pi$ with $\pi<\Ga(\mf p)\cap\Ga(\mf q)$. We will show that there is a cover $\what M\ra M$ of degree 1, 2, or 4 so that $\what M$ has a tangential map $\what M\ra S^n$, and hence $\what M$ is stably parallelizable. The group $\pi$ is a subgroup of the identity component $\SO_0(n,1)<\SO(n,1)$. The inclusions $\pi\hra\SO_0(n,1)\hra\SO_{n+1}(\C)$ define flat bundles over $M$. By Deligne--Sullivan \cite{deligne-sullivan}, there is a particular cover $\what M\ra M$ so that the map $\what M\ra M\ra B\SO_{n+1}(\C)$ is homotopically trivial. This cover is the one corresponding to the subgroup $\what\pi=\pi\cap\ker(\al_{\mf p})\cap\ker(\al_{\mf q})$ of $\pi$. Note that the index $[\pi:\what\pi]$ is 1, 2, or 4 because $\ker(\al_{\mf p})$ has index 2 in $\ker(\ov\al_{\mf p})$. Furthermore, if $n$ is even, then $\SO_{n+1}(p)< O_{n+1}(p)$ has trivial center, so $\SO_{n+1}(p)\simeq\PSO_{n+1}(p)$, which implies that $\what\pi=\pi$. 

Since there is a fibration 
\[\SO_{n+1}(\C)/\SO_0(n,1)\ra B\SO_0(n,1)\ra B\SO_{n+1}(\C)\]
and $\what M\ra B\SO_0(n,1)\ra B\SO_{n+1}(\C)$ is trivial, the map $\what M\ra B\SO_0(n,1)$ lifts to $\SO_{n+1}(\C)/\SO_0(n,1)$, which is homotopy equivalent to $\SO(n+1)/\SO(n)\simeq S^n$. This map $\what M\ra S^n$ is a tangential map by Okun \cite[\S5]{okun}. This completes the construction of the stably parallelizable cover.

%Assuming $\pi<\Ga(p)\cap\Ga(q)$, set $\what\Ga(p)=\ker\big[\Ga\ra O_{n+1}(p)\big]$, and define $\what\Ga(\mf q)$ similarly. Consider $\what\pi:=\pi\cap\what\Ga(p)\cap\what\Ga(q)$. Let $\what M\ra M$ be the cover with fundamental group $\what\pi$. It follows from Deligne--Sullivan \cite{deligne-sullivan} (see also \cite[pg.\ 553]{sullivan-stably-parallelizable}) that $\what M$ is stably parallelizable. Note that the index $[\pi:\what\pi]$ is either 1, 2, or 4.

Now we show we can find $M$ with isometry group $F$ and fundamental group $\pi<\Ga(\mf p)\cap\Ga(\mf q)$. As above, fix $\mf p\sbs\ca O_S$ such that $\alpha_p:\Lambda\ra Q_p$ is surjective and also $\alpha(\Lam_1)=\alpha(\Lam_2)=Q_p$.  

{\it Observation.} Fix a prime ideal $\mf q\sbs\ca O_S$ and denote $q\in\N$ the prime with $(q)=\mf q\cap\Z$. If the image of $\ov\al_{\mf q}:\Lam(p)\ra\PO_{n+1}(q)$ contains $Q_q$, then the image of $\ov\al_{\mf p,\mf q}: \Lambda\to \PO_{n+1}(p)\times\PO_{n+1}(q)$ defined by
\begin{equation*}
\ov\al_{\mf p,\mf q}(g)=(\ov\al_{\mf p}(g),\ov\al_{\mf q}(g))
\end{equation*}
contains $Q_p\ti Q_q$. Indeed, if $(x,y)\in Q := Q_p\times Q_q$, then one has that $\ov\al_{\mf p}(g)=x$ for some $g\in\Lambda$ and also $\ov\al_{\mf q}(h)=\ov\al_{\mf q}(g)^{-1}y$ for some $h\in\Lambda(\mf p)$. Thus $\ov\al_{\mf p,\mf q}(gh)=(x,y)$. 

We use the observation together with the strong approximation theorem to conclude that for all but finitely many of the infinitely many primes $q$ that split completely, the image of each of $\Lam$, $\Lam_1$, and $\Lam_2$ in $\PO_{n+1}(p)\ti\PO_{n+1}(q)$ contains $Q_p\ti Q_q$. As before, we may assume (by replacing $\Lam$ with a finite-index subgroup) that $\ov\al_{\mf p,\mf q}(\Lam)=Q_p\ti Q_q$. 

Set $T=\ov\al_{\mf p,\mf q}(\Lambda_3)$. The subgroup $T<Q$ has the property that there are no nontrivial $N\lhd Q$ such that $1\le N\leq T$ (compare \cite[\S3.2]{belolipetsky-lubotzky}). This holds essentially for the same reasons it holds for $T_p<Q_p$ (see \cite[\S5]{belolipetsky-lubotzky}). In our case, we only need to notice that $T\leq \PO_{n}(p)\times \PO_{n}(q)$, while the only nontrivial proper normal subgroups of $Q$ are $Q_p\times 1$ and $1\times Q_q$ (the latter fact holds because $Q_p$ and $Q_q$ are simple if $p,q$ are sufficiently large and $Q_p\not\simeq Q_q$). 

Setting $\De=\ker(\ov\al_{\mf p,\mf q})=\Lam\cap\Ga(\mf p)\cap\Ga(\mf q)$, we may repeat the argument of \cite[\S5]{belolipetsky-lubotzky} to conclude that $\pi<D$ is contained in $\Ga(\mf p)\cap\Ga(\mf q)$. This finishes part (b).

{\bf Part (c).} We explain why we can arrange for $M$ to have isometry group $F$ and arbitrarily large injectivity radius. This will follow (using Proposition \ref{prop:injrad} below) from the fact that $\pi$ is a subgroup of matrices $\Sl_m(\ca O_S)$ with coefficients in the ring $\ca O_S$ of $S$-integers in a number field $L$. Before proving Proposition \ref{prop:injrad} we recall a few facts about $\ca O_S$. Here $\ca O$ is the ring of integers in $L$, and $S$ is a finite set of places (i.e.\ an equivalence class of absolute value on $L$) that includes all of the Archimedean places, and $\ca O_S=\{x\in L: t(x)\le 1\text{ for all places }t\notin S\}$. 

For our proof of Proposition \ref{prop:injrad}, we recall the description of the set of all places of $L$. This is the content of Ostrowski's theorem \cite[Ch.\ II]{janusz}. The Archimedean places all come from embeddings of $L$ into $\R$ or $\C$. The non-Archimedean places come from prime ideals $\mf q\sbs\ca O$ as follows. Given $\mf q$, for $a\in\ca O$ define $\nu_{\mf q}(a)\in\Z_{\ge0}$ as the multiplicity of $\mf q$ appearing in the prime factorization of the ideal $(a)\sbs\ca O$; this is extended to $x=\fr{a}{b}\in L$ by $\nu_{\mf q}(x)=\nu_{\mf q}(a)-\nu_{\mf q}(b)$. Denoting the norm $N(\mf q)=|\ca O/\mf q|$, the function $t_{\mf q}(x)=N(\mf q)^{-\nu_{\mf q}(x)}$ defines a place of $L$. The set of all places (normalized in the way we have described) satisfies the \emph{product formula} $\prod t(x)=1$ for any $x\in L^\ti$ \cite[Ch.\ II, \S6]{janusz}. 
%As a consequence, if $x\in\ca O$, then $\prod_{t\in V^\infty}t(x)=N(x)$. 
For future reference, observe that if $a\in\ca O$ and $\mf q\nmid a$, then $t_{\mf q}(a)=1$, so only finitely many terms in the product $\prod t(x)$ differ from 1. Note also that if $(a)=\mf q_1^{n_1}\cd\mf q_f^{n_f}$ is the prime factorization, then $N(a)=N(\mf q_1)^{n_1}\cd N(\mf q_f)^{n_f}$, so by the product formula, $N(a)$ is also equal to the product $\prod_{t\mid\infty} t(a)$ over Archimedean places of $L$. 

\begin{prop}[Injectivity radius growth in congruence covers]\label{prop:injrad}
Let $V$ be a closed aspherical Riemannian manifold with fundamental group $\pi$. Suppose there exists an injection 
$\pi\hra\Sl_m(\ca O_S)$, where $\ca O_S$ is the ring of $S$-integers in a number field $L$. For an ideal $\mf k\sbs\ca O$, denote 
\[\Sl_m(\mf k)=\ker\big[\Sl_m(\ca O_S)\ra\Sl_m(\ca O_S/\mf k\ca O_S)\big]\] 
and let $V_{\mf k}$ be the cover of $V$ with fundamental group $\pi(\mf k):=\pi\cap\Sl_m(\mf k)$. Then there are constants $C,D$ (depending only on $V$, $m$, and $K$, but not $\mf k$) so that $\InjRad(V_{\mf k})\ge C\log k+D$, where $(k)=\mf k\cap\Z$. 
\end{prop}

This statement is similar to the ``Elementary Lemma" of \cite[\S3.C.6]{gromov-systole}. The proof below is based on, and has some overlap with, the argument in \cite[\S4]{guth-lubotzky}. 

\begin{proof}[Proof of Proposition \ref{prop:injrad}]
Let $\wtil V$ be the universal cover of $V$. 

Fix the ideal $\mf k$, and set $R=\InjRad(V_{\mf k})$.  By definition of $\InjRad$, there exists $y,z\in \wtil V$ and $\eta\in\pi(\mf k)$ so that $y,\eta y$ are both contained in the ball $B_{2R}(z)$. Then $d(y,\eta y)\le 4R$; equivalently 
\[R\ge\fr{1}{4} d(y,\eta y).\] To prove the proposition, we will give a lower bound on $d(y,\eta y)$. 

Since $V$ is compact, $\pi$ is finitely generated. Consider the generating set associated to the Dirichlet fundamental domain $\ca D$ centered at $y$ for the action of $\pi$ on $\wtil V$ (generators are those $g\in\pi$ for which $g(\ca D)\cap \ca D\neq\vn$). For the word length $w:\pi\ra\Z_{\ge0}$ associated to this generating set, there is a bound $w(\eta)\le c_1\cdot\big[d(y,\eta y)+1\big]$, obtained as follows. Take a geodesic $\ga$ connecting $y,\eta y$, and cover it by $\lfloor d(y,\eta y)\rfloor+1$ balls of radius 1. There is $c_1>0$ so that each ball intersects at most $c_1$ translates of $\ca D$, so $\ga$ intersects at most $c_1\cdot\big[d(y,\eta y)+1\big]$ translates of $\ca D$. This proves the aforementioned bound, which is equivalent to
\[d(y,\eta y)\ge (1/c_1)\cdot w(\eta)-1.\]%\le c_1\>\big[d(y,\eta y)+1\big]$. 

To finish the proof, we prove 
\begin{equation}\label{eqn:bound-word-length}
w(\eta)\ge c_2\log k+c_3\end{equation} for some constants $c_2,c_3$. Now we use the assumptions that $\pi<\Sl_m(\ca O_S)$ and $\eta\in\Sl_m(\mf k)$. For $X=(x_{ij})\in \Sl_m(L)$ and $s\in S$, define 
\[|X|_s=\max_{i,j} s(x_{ij})\>\>\>\text{ and }\>\>\> |X|_S=\sum_{s\in S}|X|_s.\] By the formula for matrix multiplication $|XY|_S\le m\>|X|_S|Y|_S$. Write $\eta=X_1\cd X_{w(\eta)}$ with $X_i\in\Sl_m(\ca O_S)$ belonging to our chosen generating set of $\pi$. Then  $|\eta|_S\le m^{w(\eta)-1}\cdot M^{w(\eta)}$, where $M$ is the maximum value of $|\cdot|_S$ on generators of $\pi$. On the other hand, we will show  that $|\eta|_S\ge \ell\cdot k^{1/\ell}-\ell$, where $(k)=\mf k\cap\Z$ and $\ell=|S|$. Then altogether we have 
\[\ell\cdot k^{1/\ell}-\ell\le |\eta|_S\le m^{w(\eta)-1}\cdot M^{w(\eta)},\] which gives a bound as in (\ref{eqn:bound-word-length}) after taking log. Note that $\log(k^{1/\ell}-1)=\log(k^{1/\ell})+\log(1-k^{-1/\ell})$ and $\log(1-k^{-1/\ell})$ is bounded below by the constant $\log(1-2^{-1/\ell})$. 

Now we prove $|\eta|_S\ge\ell\cdot k^{1/\ell}-\ell$. Since $\eta\neq\id$, some entry $\eta_{ij}$ has the form $1+x$ or $x$, where $x\in\mf k\ca O_S$ is nonzero. Write $x=\fr{a}{b}\cdot x_1$, where $x_1\in\mf k$ and the only primes dividing $a,b$ are primes in $S$. By the product formula 
\[\prod_{s\in S}s(a/b)=1\>\>\>\text{ and }\>\>\>\prod_{s\in S}s(x_1)=N(x_1).\] Furthermore, $N(x_1)\ge N(\mf k)\ge k$ because $(x_1)\sbs\mf k$ and $\Z/k\Z\sbs\ca O/\mf k$. Therefore, $\prod_{s\in S}s(x)\ge k$. 

Next we show that $\prod_{s\in S}s(x)\ge k$ implies that $|x|_S:=\sum_{s\in S} s(x)\ge \ell k^{1/\ell}$. This follows from some calculus: we want to minimize the function $\phi(x_1,\ld,x_\ell)=x_1+\cd+x_\ell$ under the constraint $x_1\cd x_\ell\ge k$. Since $\phi$ has no critical points, the minimum is achieved on the set $x_1\cd x_\ell=k$. Using Lagrange multipliers, one finds that $\phi$ has a unique minimum at $x=(k^{1/\ell},\ld,k^{1/\ell})$ and the minimum value is $\phi(x)=\ell\cdot k^{1/\ell}$. 

Since $\eta_{ij}$ is either $x$ or $1+x$, in either case $|\eta_{ij}|_S\ge\sum_{s\in S}[s(x)-1]\ge \ell\cdot k^{1/\ell}-\ell$. Combining everything we conclude that 
\[|\eta|_S\ge|\eta_{ij}|_S\ge\ell\cdot k^{1/\ell}-\ell.\]
%Next we show that 
%%To prove $|\eta|_S\ge\ell\cdot k^{1/\ell}-\ell$, first consider the case $\eta_{ij}\neq0$ for some off-diagonal entry $i\neq j$. For this entry,  
%%\[\prod_{s\in S}s(\eta_{ij})\ge N(\mf k)\ge k.\]
%%The first inequality follows from the product formula. The second inequality follows from the fact that $\Z/k\Z$ embeds in $\ca O_S/\mf k$. 
%$\prod_{s\in S}s(\eta_{ij})\ge k$ implies that $|\eta_{ij}|_S:=\sum_{s\in S}s(\eta_{ij})\ge \ell\cdot k^{1/\ell}$. This follows from some calculus: we want to minimize the function $\phi(x_1,\ld,x_\ell)=x_1+\cd+x_\ell$ under the constraint $x_1\cd x_\ell\ge k$. Since $\phi$ has no critical points, the minimum is achieved on the set $x_1\cd x_\ell=k$. Using Lagrange multipliers, one finds that $\phi$ has a unique minimum at $x=(k^{1/\ell},\ld,k^{1/\ell})$ and $\phi(x)=\ell\cdot k^{1/\ell}$. 
%Next consider the case when $\eta=\diag(\eta_1,\ld,\eta_m)$ is diagonal. Note $\eta_i\neq1$ for some $i$ because $\eta\neq\id$. For definiteness, assume $\eta_m\neq1$. Then $\eta_m=1+x$ for some $x\in\mf k$. Furthermore, $\prod_{i=1}^{m-1}\eta_i=\fr{1}{1+x}$ because $\det\eta=1$. In order for both $1+x\in\ca O_S$ and $\fr{1}{1+x}\in\ca O_S$, one of these, say $1+x$, must be in $\ca O$. It follows that $x\in\ca O$. We're going to show $|1+x|_S\ge\ell\cdot k^{1/\ell}-\ell$. As before, $\prod_{s\in S}s(x)\ge N(\mf k)\ge k$, which implies that $\sum_{s\in S} s(x)\ge \ell\cdot k^{1/\ell}$. Then 
%\[|\eta|_S\ge |1+x|_S\equiv \sum_{s\in S}s(1+x)\ge\sum_{s\in S}s(x)-1\ge \ell\cdot k^{1/\ell}-\ell.\]
This completes the proof.
\end{proof}
%if Gamma is the fundamental group of a compact manifold and Gamma embeds into SL(n,Z_S) and one considers Gamma(k) = Gamma \cap [congruence subgroup of SL(n,Z_S) by taking mod k], then the injectivity radius grows like log(k). 

\section{Symmetry constant for $N=M_{c,\phi}$}\label{sec:asymmetric}

In this section we prove Theorem \ref{thm:asymmetric}. As mentioned in the introduction, the goal is to find smooth structures $N$ and large subgroups $F<\Out(\pi)$ so that $\im\Psi_N\cap F=1$. To this end, we consider the exotic smooth structures $N=M_{c,\phi}$ studied in \cite{FJ-nonuniform}. Here $M$ is hyperbolic, $c$ is a simple closed geodesic, and $\phi\in\Diff(S^{n-2})$. Choosing a framing $\iota:S^1\times D^{n-1}\ra M$ of $c$, the manifold $M_{c,\phi}$ is defined as the quotient of 
\[S^1\ti D^{n-1} \coprod M\setminus \iota(S^1\ti\text{ int}( D^{n-1}))\]
by the identification $(x,v)\leftrightarrow \iota(x,\phi(v))$ for $(x,v)\in S^1\ti S^{n-2}$. 

We prove Theorem \ref{thm:asymmetric} in 3 steps. 

\subsection{Non-concordant smooth structures (Step 1)}

Our mechanism for constructing $\al\in\Out(\pi)$ such that $\al\notin\im\Psi_N$ is Theorem \ref{thm:non-concordant} below. Before we state it, recall some facts about smooth structures that will be used here and in the next subsection.

{\bf Smoothings of topological manifolds.} By a smooth manifold $N$ we mean a topological manifold with a smooth atlas of charts $\R^n\supset U_\alpha \ra N$ (which we call a \emph{smooth structure}). If $N$ (resp.\ $M$) is a smooth (resp.\ topological) manifold and $h:N\ra M$ is a homeomorphism, then we obtain a smooth structure on $M$ by pushforward. The map $h$ is called a \emph{marking}. Two markings $h_0:N_0\ra M$ and $h_1:N_1\ra M$ determine the same smooth  structure on $M$ if there is a diffeomorphism $g:N_0\ra N_1$ so that $h_1g=h_0$.

Two smooth structures $N_0,N_1$ on $M$ are \emph{concordant} if there exists a smooth structure on $M\times[0,1]$ whose restriction to $M\times\{i\}$ is $N_i$ for $i=0,1$. The main fact about concordances that we use is that classifying concordance classes reduces to homotopy theory: there is a bijection between the set of concordance classes of smooth structures on $M$ and the set of based homotopy classes of maps $[M,\Top/O]$. 

% are two main fact about concordance classes that we will use. 
%\begin{enumerate} 
%\item Smooth structures $N,N'$ are concordant if and only if there exists a diffeomorphism $g:N\ra N'$ that, when viewed as a homeomorphism $M\ra M$, is isotopic to the identity. 
%\item There is a bijection between the set of concordance classes of smooth structures on $M$ and the set of (based) homotopy classes of maps $[M,\Top/O]$. In other words, classifying concordance classes reduces to homotopy theory. 
%\end{enumerate} 

As remarked in \cite[\S1]{FJ-nonuniform}, the concordance class of the smooth structure $M_{c,\phi}$ is independent of the choice of framing and is also independent of the choice of representative of the isotopy class $[\phi]\in\pi_0\Diff(S^{n-2})$. 

\begin{thm}[non-concordant smooth structures]\label{thm:non-concordant}
Let $M$ be a smooth closed manifold. Assume $M$ is stably parallelizable. Let $c_1,\ld,c_\ell$ be disjoint closed curves in $M$. Assume that there exists a homomorphism $\De:\pi_1(M)\ra\Z^\ell$ such that $\De(c_1),\ld,\De(c_\ell)$ generate $\Z^\ell$. For any nontrivial isotopy class $[\phi]\in\pi_0\Diff(S^{n-2})$, no two of the smooth structures $M_{c_1,\phi},\ldots,M_{c_\ell,\phi}$ are concordant. 
\end{thm}

\begin{proof}
Given a codimension-0 embedding $\la:X\ra Y$ of open manifolds, we denote $\la'$ the induced map of 1-point compactifications, obtained by collapsing $Y\setminus X$ to a point.  Also $X_+$ denotes the space $X$ with a disjoint basepoint.

Let $\iota_1,\ld,\iota_\ell:S^1\times D^{n-1}\hra M$ be framings of $c_1,\ld,c_\ell$. Use $\iota_1,\ldots,\iota_\ell$ to define an embedding $\iota:\coprod_\ell S^1\ti D^{n-1}\hra M$. The induced collapse map has the form $\iota':M\ra\bigvee_\ell \Si^{n-1}(S^1_+)$. 
%which we extend to a map $\iota':M_+\ra\bigvee_\ell \Si^{n-1}(S^1_+)$. 
Consider the composition 
\[\hat\iota:M_+\ra M\xra{\iota'}\bigvee_\ell\Si^{n-1}(S^1_+)\ra \bigvee_\ell S^{n-1},\] where the last map is induced from the obvious maps $\Si^{n-1}(S^1_+)\simeq S^n\vee S^{n-1}\ra S^{n-1}$. It suffices to show that the induced map 
\[\hat\iota^*: \big[\bigvee_\ell S^{n-1},\>\Top/O\big]\ra\big[M_+,\>\Top/O\big]\]
is injective. This is because, under the bijection between concordance classes of smooth structures on $M$ and $[M,\Top/O]$, the concordance class of $M_{c_j,\phi}$ corresponds to the map 
%\[\psi_j: 
\[M\xra{\hat\iota} \bigvee_\ell S^{n-1}\xra{\pi_j} S^{n-1}\xra{\hat\phi} \Top/O,\]
%\[M\xra{\iota'} \bigvee_\ell \big[(S^1)_+\we S^{n-1}\big]\ra \bigvee_\ell S^{n-1},\]
%which induces a map $j:\prod_\ell [S^{n-1},\Top/O]\simeq[\bigvee_\ell S^{n-1},\Top/O]\ra [M,\Top/O]$. 
where $\pi_j$ collapses every sphere other than the $j$-th sphere to the basepoint, and $\hat\phi$ corresponds to $[\phi]\in\pi_0\Diff(S^{n-2})$ under the bijections $[S^{n-1},\Top/O]\simeq \Theta_{n-1}\simeq\pi_0\Diff(S^{n-2})$. %; the map ; and remaining map is induced by $\Si^{n-1}(S^1_+)\simeq S^n\vee S^{n-1}\ra S^{n-1}$. %the map $\pi_i:\bigvee_\ell S^{n-1}\ra S^{n-1}$ collapses every sphere other than the $i$-th sphere to the basepoint; and the map $\hat\phi$ corresponds to $[\phi]$ under the bijections $[S^{n-1},\Top/O]\simeq \Theta_{n-1}\simeq\pi_0\Diff(S^{n-2})$. %The concordance class defined by $\psi_j$ is $M_{\iota_j,\phi}$. 

%To prove the theorem, we show that no two of the maps $\psi_j$ are homotopic. 

To show that $\hat\iota^*$ is injective, we use that $\Top/O$ is an infinite loop space. In particular, there exists a space $Y$ such that $\Om^{n+\ell}Y\simeq\Top/O$, and for any space $A$, there are natural bijections $[A,\Top/O]\simeq[A,\Om^{n+\ell}Y]\simeq[\Si^{n+\ell}A,Y]$. This allows us to view $\hat\iota^*$ as map 
\[\big[\bigvee_\ell S^{2n+\ell-1},Y
\big]\ra\big[\Si^{n+\ell}(M_+),Y]. \]
This map can also be obtained by considering the embedding $\iota\times 1:\left(\coprod_\ell S^1\ti D^{n-1}\right)\times D^{n+\ell}\hra M\times D^{n+\ell}$ and the composition 
$\what{\iota\ti 1}:\Si^{n+\ell}(M_+)\xra{(\iota\times1)'} \bigvee_\ell\Si^{2n+\ell}(S^1_+)\ra\bigvee_\ell S^{2n+\ell-1}$, similar to before.  

%Since $\Top/O$ is an infinite loop space, there exists $Y$ such that $\Om^{n+\ell}Y\simeq\Top/O$. Abusing notation, let $\psi_j:\Si^{n+\ell}M\ra Y$ be the map corresponding to $\psi_j:M\ra\Top/O$ under the adjunction $[M,\Top/O]\simeq[M,\Om^{n+\ell}Y]\simeq [\Si^{n+\ell}M,Y]$. It suffices to show that no two of the maps $\psi_j:\Si^{n+\ell}M\ra Y$ are homotopic. 

The homomorphism $\De$ is induced by a map $\de:M\ra T^\ell$ to the torus, and we can assume $\de$ is smooth. Take a Whitney embedding $\ep:M\ra D^{2n}$, and consider the induced embedding $\de\ti\ep:M\ra T^\ell\ti D^{2n}$. Since $M$ is a stably parallelizable, $M\sbs T^\ell\ti D^{2n}$ has trivial normal bundle $\nu_M\simeq\ep^{n+\ell}$.  (To see this, observe that $TM\op\nu_M\simeq\ep^{2n+\ell}$. Since $M$ is stably parallelizable, $TM\oplus\ep\simeq\ep^{n+1}$, which implies that $\ep^{n+1}\oplus\nu_M\simeq\ep^{2n+\ell+1}$. Since $\text{rank}(\nu_M)>\dim M$, this implies that $\nu_M$ is the trivial bundle by \cite[Lem.\ 3.5]{kervaire-milnor}.) Then there is an embedding $\ka:M\ti D^{n+\ell}\ra T^\ell\ti D^{2n}$. 

Consider now the composition 
\[p:\label{eqn:torus-collapse}\Si^{2n}(T^\ell_+)\xra{\ka'} \Si^{n+\ell}(M_+)\xra{\what{\iota\times1}} \bigvee_\ell S^{2n+\ell-1}.\]
%We'll show that no two of these maps are homotopy equivalent. 
To prove the theorem, we show that the induced map 
\[p^*:\big[\bigvee_\ell S^{2n+\ell-1},Y\big]\ra\big[\Si^{2n}(T^\ell_+),Y\big]\]
is injective. First observe the homotopy equivalence $\Si^{2n}(T^\ell_+)\sim \bigvee_{i=0}^\ell{\ell\choose i}S^{2n+i}$. 
%\begin{equation}\label{eqn:homotopy-suspension}\Si^{2n}(T^\ell_+)\sim \bigvee_{i=0}^\ell{\ell\choose i}S^{2n+i}\>\>\>\text{ and }\>\>\>\bigvee_\ell \Si^{2n+\ell-1}(S^1_+)\sim \ell(S^{2n+\ell})\vee\ell(S^{2n+\ell-1}).\end{equation}
This follows from general homotopy equivalences $\Si(A_+)\sim \Si A\vee S^1$ and $\Si(A\times B)\sim \Si A\vee\Si B\vee \Si(A\wedge B)$. Since $\De(c_1),\ldots,\De(c_\ell)$ generate $\pi_1(T^\ell)$, the inclusion $\ell\> S^{2n+\ell-1}\sbs\bigvee_{i=0}^\ell{\ell\choose i}S^{2n+i}$ is a right inverse to $p$, up to homotopy. This implies that $p^*$ is injective. 
%{\color{red} Fill: we know exactly how this composition works: corresponding to the map $S^{n-1}\ra\Top/O$ is a map $S^{2n+\ell-1}\ra Y$ and the maps $(T^\ell)_+\we S^{2n}\simeq\Si^{2n}(T^\ell)_+\ra Y$ factors $(T^\ell)_+\we S^{2n}\ra (S^1)_+\we S^{\ell-1}\we S^{2n}\simeq (S^1)_+\we S^{2n+\ell-1}\ra S^{2n+\ell-1}$, where the first map is the collapse map associated to one of the 1-dimensional homology classes...} 
\end{proof}

\subsection{Outer automorphisms not realized by diffeomorphisms (Step 2)}

Next we apply Theorem \ref{thm:non-concordant} to give a criterion that guarantees that $\al\in\Out(\pi)$ is not in the image of $\Psi_N:\Diff(N)\ra\Out(\pi)$. 

%Before this, we establish some facts we'll need about isotopy classes of homeomorphisms and diffeomorphisms of $M$. 

%\begin{prob}\label{prob:isotopy-classes}
%Let $\Homeo_0(M)<\Homeo(M)$ be the subgroup of homeomorphisms homotopic to the identity. This group is the kernel of $\Homeo(M)\ra\Out(\pi)$. Define $\Diff_0(M)$ similarly. 
%\begin{enumerate}
%\item Compute $\pi_0\Homeo_0(M)$ and $\pi_0\Diff_0(M)$. 
%\item Conclude/show that $\pi_0\Diff_0(M)\ra\pi_0\Homeo_0(M)$ is surjective and that the kernel comes from Dehn twists.. 
%\item Show that every element of $\pi_0\Homeo_0(M)$ is represented by a homeomorphism with ``small support" and can also be isotoped to a diffeomorphism with small support. 
%\end{enumerate} 
%\end{prob}

\begin{thm}[obstruction to Nielsen realization]\label{thm:non-realize}
Let $M$ be a hyperbolic manifold and fix a simple closed geodesic $c$ in $M$. Let $N=M_{c,\phi}$ be an exotic smooth structure. Assume that $\al\in\Isom(M)\simeq\Out(\pi)$ is such that $M_{c,\phi}$ and $M_{\al(c),\phi}$ are not concordant. Then $\al\notin\im\Psi_N$. 
\end{thm}

\begin{proof}
Suppose for a contradiction that there is a diffeomorphism $f:N\ra N$ such that $\Psi_N(f)=\al$. 

Set $N_0=N$ and $N_1=M_{\al(c),\phi}$, and observe that $\al:M\ra M$ induces a diffeomorphism $g_1:N_0\ra N_1$. Define $g_2=g_1\circ f^{-1}$. Denoting $h_i:N_i\ra M$ be the obvious homeomorphisms, the composition 
\[M\xra{h_0^{-1}} N_0\xra{g_2} N_1\xra{h_1}M\]
induces the identity on $\pi$ and is therefore homotopic to the identity. From this homotopy, we obtain a homotopy equivalence $H_0:M\ti[0,1]\ra M\ti[0,1]$, which restricts to a homeomorphism on the boundary. By \cite[Cor.\ 10.6]{FJ-mostow}, $H_0$ is homotopic rel boundary to a homeomorphism $H$. Then the composition 
\[N_0\times[0,1]\xra{h_0\ti\text{id}}M\ti[0,1]\xra{H}M\ti[0,1]\]
defines a smooth structure on $M\times[0,1]$ whose restriction to $M\ti\{i\}$ is $N_i$ for $i=0,1$, i.e.\ $N_0$ and $N_1$ are concordant. This contradicts our assumption, so $\alpha\notin\im\Psi_N$. 
\end{proof}

%
%\begin{proof} %This is the same proof but with the old notation
%Suppose for a contradiction that there is a diffeomorphism $\be:N\ra N$ such that $\Psi_N(\be)=\al$. 
%
%Set $N'=M_{\al f,\phi}$, and observe that $\al:M\ra M$ induces a diffeomorphism $\hat\al:N\ra N'$. The composition 
%\[\ga:N\xra{\be^{-1}} N\xra{\hat\al} N'\] induces the identity on $\pi$ and is therefore homotopic to the identity $\bbm 1_M$ (here we use the obvious homeomorphisms $N\simeq M\simeq N'$ to view $\ga\in\Homeo(M)$ and to identify $\pi_1(N)\simeq \pi\simeq\pi_1(N')$). 
%
%Next we show that we can modify $\ga$ to obtain a diffeomorphism $\ga':N\ra N'$ that is isotopic to $\bbm 1_M$ through homeomorphisms. Since isotopy implies concordance for smooth structures, then $N$ and $N'$ are concordant. This is our desired contradiction. 
%
%Take $\de\in\Homeo_0(M)$ with $[\de]=[\ga]^{-1}$ in $\pi_0\Homeo_0(M)$. By the solution to Problem \ref{prob:isotopy-classes} Part (2), we may assume that $\de$ is smooth. 
%
%
%{\it My guess from our conversations is that the kernel of $\pi_0\Diff_0(M)\ra\pi_0\Homeo_0(M)$ comes from Dehn twists about totally geodesic submanifolds. We want to choose $\de\in\Diff_0(M)$ that doesn't do any Dehn twisting..}
%
%By the solution to Problem \ref{prob:isotopy-classes} Part (3), we can assume that the support of $\de$ is disjoint from $f(S^1\ti D^n)$, and so $\de$ induces a diffeomorphism $\de':N\ra N$. By construction $\ga\circ\de':N\ra N\ra N'$ is a diffeomorphism that is, as a homeomorphism of $M$, isotopic to the identity. 
%\end{proof}

\subsection{Examples (Step 3)} To complete the proof of Theorem \ref{thm:asymmetric}, we explain how to obtain examples of stably parallelizable $M$ that satisfy the assumptions of Theorems \ref{thm:non-concordant} and \ref{thm:non-realize}. This is the content of the following proposition. 

\begin{prop}\label{prop:examples}
Fix $n\ge2$. For any $d\ge2$, there exists a stably parallelizable hyperbolic manifold $M^n$, a geodesic $c$, a subgroup $F<\Isom(M)$ isomorphic to $\Z/d\Z=\pair{\al}$, and $\rho\in H^1(M)\simeq\Hom(H_1(M),\Z)$ such that 
\begin{equation}\label{eqn:homomorphism}\rho(\al^jc)=\begin{cases}1&j=0\\0&1\le j\le d-1.\end{cases}\end{equation} Consequently, the homomorphism $\De:H_1(M)\ra\Z^d$ whose $i$-th coordinate is $\rho\circ\al^{-i}$ has the property that $\De(c),\ld,\De(\al^{d-1}c)$ generate $\Z^d$. 
\end{prop} 

In \cite{lubotzky-betti}, Lubotzky gave examples of hyperbolic $M$ (both arithmetic and non-arithmetic) with a surjection $\pi_1(M)\onto F_r$ to a free group of rank $r\ge2$. By passing to a cover, we can assume that $M$ is stably parallelizable \cite[pg.\ 553]{sullivan-stably-parallelizable}. Proposition \ref{prop:examples} is proved by passing to a further cover, using the general procedure of the following lemma.

\begin{lem}\label{lem:covers-betti}
Let $X$ be a CW-complex, and let $F_r$ denote a free group of rank $r\ge2$. Assume there is a surjection $\pi_1(X)\onto F_r$. Then for any $d\ge2$, there exists a regular cover $Y\ra X$ with deck group $\Z/d\Z=\pair{\al}$ and $c\in\pi_1(Y)$ and $\rho\in H^1(Y)$ satisfying (\ref{eqn:homomorphism}). 
%$\rho(c)=1$, and so that if $\De:H_1(Y)\ra\Z^d$ is the homomorphism with $i$-th coordinate $\rho\circ\al^{-i}$, then $\De(c),\ld,\De(\al^{d-1}c)$ generate $\Z^d$. 
\end{lem}

\begin{proof}
Take $F_r$ with generators $a_1,\ld,a_r$. Consider $F_r\onto\Z/d\Z$ defined by $a_1\mapsto 1$ and $a_i\mapsto 0$ for $2\le i\le r$. Then $\ker[F_r\onto\Z/d\Z]\simeq F_k$ with $k=1+d(r-1)$. It's easy to compute $H_1(F_k)$ as a $F=\Z/d\Z$-module:
\[H_1(F_k)\simeq \Z\{b_1\}\oplus\Z F\{b_2,\ld,b_k\}.\]
(For example, realize $1\ra F_k\ra F_r\ra\Z/d\Z\ra 0$ as a $(\Z/d\Z)$-covering of graphs.) Then also $H^1(F_k)\simeq\Z\{\be_1\}\oplus\Z F\{\be_2,\ld,\be_k\}$, where $\be_i$ is dual to $b_i$. 

Let $Y\ra X$ be the cover such that $\pi_1(Y)=\ker\big[\pi_1(X)\onto F_r\onto\Z/d\Z\big]$. Then $\pi_1(Y)\onto F_k$, and $H_1(Y)\ra H_1(F_k)$ is $(\Z/d\Z)$-equivariant. Choose $c\in\pi_1(Y)$ so that $c\mapsto b_2$ under $\pi_1(Y)\onto F_k$, and define $\rho:\pi_1(Y)\onto F_k\xra{\be_2}\Z$. It's easy to verify that $\rho$ satisfies (\ref{eqn:homomorphism}). 
%Observe that $(\rho\circ\al^{-i})(\al^jc)=\begin{cases}1&i=j\\0&i\neq j\end{cases}$. 
This proves the lemma.
\end{proof}

\bibliographystyle{alpha}
\bibliography{nielsen-bib}

\begin{thebibliography}{BRHW17}

\bibitem[AD02]{adem-davis}
A.~Adem and J.~F. Davis.
\newblock Topics in transformation groups.
\newblock In {\em Handbook of geometric topology}, pages 1--54. North-Holland,
  Amsterdam, 2002.

\bibitem[BL05]{belolipetsky-lubotzky}
M.~Belolipetsky and A.~Lubotzky.
\newblock Finite groups and hyperbolic manifolds.
\newblock {\em Invent. Math.}, 162(3):459--472, 2005.

\bibitem[Bor83]{borel-isometry-aspherical}
Armand Borel.
\newblock On periodic maps of certain {K}($\pi$,1).
\newblock In {\em \OE uvres: collected papers. {V}ol. {I}}, pages 57--60.
  Springer-Verlag, Berlin, 1983.

\bibitem[BRHW17]{brhw}
A.~Brown, F.~Rodriguez~Hertz, and Z.~Wang.
\newblock Global smooth and topological rigidity of hyperbolic lattice actions.
\newblock {\em Ann. of Math. (2)}, 186(3):913--972, 2017.

\bibitem[BW08]{block-weinberger}
J.~Block and S.~Weinberger.
\newblock On the generalized {N}ielsen realization problem.
\newblock {\em Comment. Math. Helv.}, 83(1):21--33, 2008.

\bibitem[DS75]{deligne-sullivan}
P.~Deligne and D.~Sullivan.
\newblock Fibr\'es vectoriels complexes \`a groupe structural discret.
\newblock {\em C. R. Acad. Sci. Paris S\'er. A-B}, 281(24):Ai, A1081--A1083,
  1975.

\bibitem[Far02]{farrell-trieste}
F.~T. Farrell.
\newblock The {B}orel conjecture.
\newblock In {\em Topology of high-dimensional manifolds, {N}o. 1, 2
  ({T}rieste, 2001)}, volume~9 of {\em ICTP Lect. Notes}, pages 225--298. Abdus
  Salam Int. Cent. Theoret. Phys., Trieste, 2002.

\bibitem[FJ89a]{FJ-exotic-hyperbolic}
F.~T. Farrell and L.~E. Jones.
\newblock Negatively curved manifolds with exotic smooth structures.
\newblock {\em J. Amer. Math. Soc.}, 2(4):899--908, 1989.

\bibitem[FJ89b]{FJ-mostow}
F.~T. Farrell and L.~E. Jones.
\newblock A topological analogue of {M}ostow's rigidity theorem.
\newblock {\em J. Amer. Math. Soc.}, 2(2):257--370, 1989.

\bibitem[FJ90]{FJ-nielsen}
F.~T. Farrell and L.~E. Jones.
\newblock Smooth nonrepresentability of {${\rm Out}\, \pi_1M$}.
\newblock {\em Bull. London Math. Soc.}, 22(5):485--488, 1990.

\bibitem[FJ93]{FJ-nonuniform}
F.~T. Farrell and L.~E. Jones.
\newblock Nonuniform hyperbolic lattices and exotic smooth structures.
\newblock {\em J. Differential Geom.}, 38(2):235--261, 1993.

\bibitem[FKS13]{fks-anosov}
D.~Fisher, B.~Kalinin, and R.~Spatzier.
\newblock Global rigidity of higher rank {A}nosov actions on tori and
  nilmanifolds.
\newblock {\em J. Amer. Math. Soc.}, 26(1):167--198, 2013.
\newblock With an appendix by James F. Davis.

\bibitem[GL14]{guth-lubotzky}
L.~Guth and A.~Lubotzky.
\newblock Quantum error correcting codes and 4-dimensional arithmetic
  hyperbolic manifolds.
\newblock {\em J. Math. Phys.}, 55(8):082202, 13, 2014.

\bibitem[GPS88]{gps}
M.~Gromov and I.~Piatetski-Shapiro.
\newblock Nonarithmetic groups in {L}obachevsky spaces.
\newblock {\em Inst. Hautes \'Etudes Sci. Publ. Math.}, (66):93--103, 1988.

\bibitem[Gro96]{gromov-systole}
M.~Gromov.
\newblock Systoles and intersystolic inequalities.
\newblock In {\em Actes de la {T}able {R}onde de {G}\'eom\'etrie
  {D}iff\'erentielle ({L}uminy, 1992)}, volume~1 of {\em S\'emin. Congr.},
  pages 291--362. Soc. Math. France, Paris, 1996.

\bibitem[HH65]{hsiang-hsiang}
W.-C. Hsiang and W.-Y. Hsiang.
\newblock Classification of differentiable actions of {$S^{n}$}, {$R^{n}$}, and
  {$D^{n}$} with {$S^{k}$} as the principal orbit type.
\newblock {\em Ann. of Math. (2)}, 82:421--433, 1965.

\bibitem[HH69]{hsiang-hsiang-degree-symmetry}
W.-C. Hsiang and W.-Y. Hsiang.
\newblock The degree of symmetry of homotopy spheres.
\newblock {\em Ann. of Math. (2)}, 89:52--67, 1969.

\bibitem[HHR16]{hill-hopkins-ravenel}
M.~A. Hill, M.~J. Hopkins, and D.~C. Ravenel.
\newblock On the nonexistence of elements of {K}ervaire invariant one.
\newblock {\em Ann. of Math. (2)}, 184(1):1--262, 2016.

\bibitem[Hsi67]{hsiang}
W.-Y. Hsiang.
\newblock On the bound of the dimensions of the isometry groups of all possible
  riemannian metrics on an exotic sphere.
\newblock {\em Ann. of Math. (2)}, 85:351--358, 1967.

\bibitem[Jan96]{janusz}
G.~J. Janusz.
\newblock {\em Algebraic number fields}, volume~7 of {\em Graduate Studies in
  Mathematics}.
\newblock American Mathematical Society, Providence, RI, second edition, 1996.

\bibitem[KM63]{kervaire-milnor}
M.~A. Kervaire and J.~W. Milnor.
\newblock Groups of homotopy spheres. {I}.
\newblock {\em Ann. of Math. (2)}, 77:504--537, 1963.

\bibitem[Lub96]{lubotzky-betti}
A.~Lubotzky.
\newblock Free quotients and the first {B}etti number of some hyperbolic
  manifolds.
\newblock {\em Transform. Groups}, 1(1-2):71--82, 1996.

\bibitem[MS74]{milnor-stasheff}
J.~Milnor and J.~Stasheff.
\newblock {\em Characteristic classes}.
\newblock Princeton University Press, Princeton, N. J.; University of Tokyo
  Press, Tokyo, 1974.
\newblock Annals of Mathematics Studies, No. 76.

\bibitem[MT18]{mann-tshishiku}
K.~Mann and B.~Tshishiku.
\newblock {Realization problems for diffeomorphism groups}.
\newblock https://arxiv.org/abs/1802.00490, February 2018.

\bibitem[Oku01]{okun}
B.~Okun.
\newblock Nonzero degree tangential maps between dual symmetric spaces.
\newblock {\em Algebr. Geom. Topol.}, 1:709--718, 2001.

\bibitem[Sul79]{sullivan-stably-parallelizable}
D.~Sullivan.
\newblock Hyperbolic geometry and homeomorphisms.
\newblock In {\em Geometric topology ({P}roc. {G}eorgia {T}opology {C}onf.,
  {A}thens, {G}a., 1977)}, pages 543--555. Academic Press, New York-London,
  1979.

\end{thebibliography}

\end{document}